\documentclass[11pt]{article}
 \usepackage[latin1]{inputenc}
\usepackage{tikz}
\usetikzlibrary{shapes,arrows}
\usepackage{fullpage}
\usepackage{amsthm} 
\usepackage{amsmath} 
\usepackage{amsfonts}  
\usepackage{amssymb}
\usepackage{comment}
\usepackage{ifpdf}
\usepackage{tabto}

\ifpdf
   \usepackage{url}
    \usepackage[pdftex,
bookmarks=true,
bookmarksnumbered=true,
hypertexnames=false,
breaklinks=true,
linkbordercolor={0 0 1},
pdfborder={0 0 1}]{hyperref}
 \else
   \usepackage{graphicx}
 \fi

\newcommand{\ch}{\mbox {\bf 1}}

\newcommand{\Z}{{\mathbb{Z}}} 
\newcommand{\N}{{\mathbb{N}}}

\newcommand{\ol}[1]{\overline{#1}}

\newcommand{\be}{\begin{equation}}
\newcommand{\ee}{\end{equation}}

\newtheorem{theorem}{Theorem}  
\newtheorem{prop}[theorem]{Proposition}  
\newtheorem{lemma}[theorem]{Lemma}
\newtheorem{corollary}[theorem]{Corollary}

\numberwithin{equation}{section}
\title{Efficient Coupling for Random Walk with Redistribution}
\author{Iddo Ben-Ari\footnote{Corresponding author. \url{iddo.ben-ari@uconn.edu}.}
\footnote{ Partially supported by NSA grant H98230-12-1-0225 and by grant  282912 from the Simons Foundation.}
 \and Hugo Panzo\footnote{Partially supported by NSA grant H98230-14-1-0134 to Richard Bass.} \and Elizabeth Tripp\footnote{Partially supported by UConn SURF Trimble Family Award.}}

\date{}
\begin{document}
\maketitle
\begin{abstract}
What can one say on convergence to stationarity of a  finite state Markov chain that behaves ``locally" like a nearest neighbor random walk on $\Z$ ? 
The model we consider is a version of nearest neighbor lazy random walk on the  state space $ \{0,\dots,N\}$: the  probability for staying put at each site is $\frac 12$,  the transition to the nearest neighbors, one on the right and one on the left, occurs with probability $\frac14$ each, where we identify two sites, $J_0$ and $J_N$ as, respectively, the neighbor of $0$ from the left and the neighbor of $N$ from the right (but $0$ is not a neighbor of $J_0$ and $N$ is not neighbor of $J_N$). This model is a discrete version of diffusion with redistribution on an interval studied by several authors in the recent past, and for  which the exponential rates of convergence to stationarity was computed analytically, but had no intuitive or probabilistic interpretation, except for case where the jumps from the endpoints are identical (or more generally have the same distribution).  We study convergence to stationarity probabilistically, by finding an efficient coupling. The coupling identifies  the ``bottlenecks" responsible for the rates of convergence and also gives tight computable bounds on the total variation norm of the process between two starting points.  The adaptation to the diffusion case is straightforward.
\end{abstract}
\section{Introduction}
The goal of this paper is to construct efficient coupling for a  discrete version of one-dimensional diffusion with redistribution on an interval, a special case of  model which was studied independently by several groups of authors (\cite{GK} \cite{BAP} \cite{KW}\cite{KW2} \cite{LLR} \cite{LP}\cite{BAd}). Our coupling gives a probabilistic explanation to the rates of convergence to stationarity for the model and partially answering an open problem posed in \cite{KW}.  The fact that the coupling does capture the rate of convergence is nontrivial. \\  

Grigorescu and Kang \cite{GK} first considered a model they called Brownian Motion on the Figure Eight. The model considered was Brownian motion on an interval $(0,1)$, which upon hitting the boundary, starts afresh at a point $a\in (0,1)$. It was shown that the model converges exponentially fast to stationarity with the convergence rate coinciding with the second eigenvalue of the Dirichlet Laplacian $-\frac 12 \frac{d^2}{dx^2}$ on $(0,1)$. The model was generalized to higher dimension \cite{GK2},  more general diffusion and boundary behavior in \cite{BAP}\ \cite{BAP2}\cite{LP}. However, the one-dimensional model with BM as the underlying diffusion (possibly with constant drift), but more general boundary behavior than BM on the Figure Eight has attracted  some attention because despite its apparent simplicity (and obvious regeneration structure), it exhibits interesting and nontrivial behavior \cite{LLR} \cite{KW}\cite{KW2}\cite{BAd}.  By more general boundary behavior we mean that upon hitting the boundary, the process starts afresh in the domain, but with an initial distribution $\nu_-$ if exiting from the left, and $\nu_+$ if exiting from the right. Of course, this mechanism is repeated indefinitely.  It was shown in \cite{LLR} that if $\nu_-=\nu_+$ then the rate of convergence is the second eigenvalue of the Dirichlet Laplacian, and that the  first Dirichlet eigenvalue is an (unattainable) infimum of rates of convergence over all choices of $\nu_-$ and $\nu_+$. In an unpublished work due to the tragic death of Wenbo Li,  it was  shown that the third Dirchlet eigenvalue is the maximal rate of convergence. All three results were obtained by Fourier analysis, and did not provide any insight on the probabilistic mechanism that governs the rate of convergence. Kolb and Wubker \cite{KW} obtained an efficient coupling for the case $\nu_-=\nu_+$, giving a beautiful and intuitive explanation to the rate of convergence, utilizing the fact that $\nu_-=\nu_+$ allows to guarantee coupling  once both copies are redistributed at the same time. This principle does not hold when $\nu_-\ne \nu_+$, a problem left open in \cite{KW}, and is the  main motivation for the present work. Although our main interest is in this latter case, we also provide our version of  the coupling in \cite{KW} to the discrete setting, as this leads to more questions, and completes the picture for the discrete setting. \\ 

We now describe our model. Our process builds from lazy random walk on the state space ${\cal S}_N=\{0,\dots,N\}$ for some $N>2$. Let $\nu_0$ be a probability distribution on $\{2,\dots,N\}$ and let $\nu_{N}$ be a probability distribution on $\{0,\dots,N-2\}$. Slightly abusing notation, we consider  $\nu_0,\nu_N$ also as the probability mass functions with domain $\Z$, through the identification  $\nu_x (z) = \nu_x (\{z\})$.  Consider the transition function $p$ on the state space given by: 
\be 
\label{eq:pdefn}
p(x,y) =   \begin{cases} \frac 12 & x=y\\  \frac 14 & |x-y|=1 \\ \frac 14 \nu_x (y) & x\in \{0,N\},|y-x|>1.
\end{cases} 
\ee 

As is easy to see, the model is never reversible.  In the sequel, we will always make the following additional assumptions, which will simplify our arguments, and allow us to focus on ideas, and less on parity-related technicalities (which are still unavoidable, but are more manageable).

\be 
\label{eq:parity_check}
\begin{aligned} 
&N \in 4\N\mbox{, and }\nu_0,\nu_N\mbox{ are both supported on  }\{3,5,\dots, N-3\}.
\end{aligned} 
\ee

As we wish to  construct a coupling and consequent bounds which are uniform under scaling as $N\to\infty$,  these assumptions pose no restriction. Any probability distribution on $(0,1)$ is a weak limit of scalings of $\nu_0$ and $\nu_N$ as above.\\

Let $X=(X_n:n\in \Z_+)$ denote the canonical process on $\Z$, and for $x \in {\cal S}_N$, let $P_x$  be the distribution  under which $X$ is a Markov chain with transition function $p$ as in \eqref{eq:pdefn}, with  initial distribution $X_0=x$. Below,  even if the initial distribution is not specified, we will assume that $X$ is a Markov chain with transition function $p$. It is easy to see that $X$ is aperiodic and irreducible. Hence it is ergodic, and in particular, it follows that it converges exponentially fast in total variation to its unique stationary distribution $\pi$. To make this more precise, for $x,y\in {\cal S}_N$, let 

\be d_t (x,y)= \max_{A \subset {\cal S}_N}  \left (  P_x (X_t \in A) - P_y (X_t \in A)\right) 
\ee 
 denote the total variation distance between $P_x ( X_t \in \cdot)$ and $P_y (X_t \in \cdot)$, and let  $d_t = \sup_{x,y} d_t(x,y)$. Then the exponential ergodicity is the statement that there exists $\lambda \in (0,1)$ such that 

\be  
\lim_{t\to \infty}  \frac 1t \ln d_t = \ln \lambda. 
\ee 

We now recall the notion of a coupling and of an efficient coupling. A  coupling for $p$ is an ${\cal S}_N^2$-valued process  $(X,Y) = \{(X_n,Y_n):n\in \Z_+\}$ such that both marginal processes $X=(X_n:n\in \Z_+)$ and $Y=(Y_n\in \Z_+)$ are Markov chains with transition function $p$. As usual, we will only consider coalescing couplings:  If $X_n = Y_n$ for some $n\in\Z_+$, then $X_k = Y_k $ for all $k\ge n$.  A coupling is Markovian if $(X,Y),X,Y $ are all Markov chains with respect to the filtration generated by $(X,Y)$. In particular, this implies  that  for  all $n\in\Z_+,~x,y,x',y' \in {\cal S}_N$,  
$$P(X_{n+1} = x' | (X_n,Y_n)=(x,y) ) = p(x,x')\mbox{ and } P(Y_{n+1} = y' | (X_n,Y_n) = (x,y)) = p(y,y').$$ 

  Given a coupling $(X,Y)$, we will write $P_{x,y}$ for the law of $(X,Y)$ with $X_0=x,Y_0=y$. Also, the coupling time $\tau$ is a stopping time defined  as 
\be \tau = \inf\{ t: X_t = Y_t\}\ee 

As is well known, 
\be d_t(x,y)  \le  P_{x,y} (\tau >t ). \ee 
Note that the righthand side depends on the specifics of the coupling. We say that a coupling is efficient if  

\be \frac 1t \ln P_{x,y} (\tau>t) \sim \frac 1t \ln d_t(x,y),\ee
as $t\to\infty$, where $f(t)\sim g(t)$ means that the ratio $f(t)/g(t)\to 1$ as $t\to\infty$. We comment that this definition is slightly weaker than the standard definition in \cite{BK}, as we require asymptotic equivalence at the logarithmic scale. The reason for that is that in some cases, the coupling constructed has a polynomial correction to the exponential decay, an effect which vanishes at the logarithmic scale.\\

We close this section with two comments: 

\begin{enumerate} 
\item As is well, known (\cite[Corollary 2.8]{BK}) efficient coupling for the similar nearest neighbor processes on  $\{0,\dots,N\}$ (birth and death processes) is trivial as any coupling (which does not allow crossings from $(x,x+1)$ to $(x+1,x)$) is efficient, do the the existence of monotone functions which separate the two copies up to the coupling time, due to the linear order of the set. Without this linear order, even in the next simplest case of nearest neighbor reversible Markov chains on the circle efficient couplings do not always exists  (although interestingly enough, an efficient coupling implies the existence of a nearly monotone function for many chains on the circle) \cite{MountCran}. 

\item As the adaptation our results from Section \ref{sec:deterministic} to the diffusion setting is straightforward and we will omit the details. 
\end{enumerate} 

\section{Our Results}
\subsection{Assumptions and Preparation} 

In order to make the main argument simpler and more  visible, we will make a reduction to a smaller set of initial distributions. To this end, we need some definitions. Let \be \rho = 2 \lfloor \frac { \min \{x,N-x : \nu_0(x)+\nu_N(x)>0\}}{2}\rfloor.\ee
 That is, $\rho$ is the largest  even number less than or equal to the distance of the  union of the support of  $\nu_0$ and $\nu_N$  to $\{0,N\}$. Observe that by our assumption \eqref{eq:parity_check},  $\rho \ge 2$. Let 

\be \tilde d_t = \sup_{y-x \in 2\N, y-x \le \rho }  d_t (x,y).\ee 
Then we have the following simple proposition: 

\begin{prop}
\label{pr:triangle} 
\be  \tilde d_t \le d_t \le \lfloor 1+ N/\rho\rfloor (\tilde d_t + \tilde d_{t-1}).\ee
\end{prop} 

\begin{proof}
The first inequality is trivial. We turn to the second. \\

From the triangle inequality,  for any  $x,y$ we have 

\be
\label{eq:triangle} 
d_t(x,y) \le d_t (x_0,x_1) + \dots + d_t (x_{n-1},x_n),
\ee 

whenever  $x=x_0< x_1< \dots < x_n = y$. Note that  $y-x = m \rho + b$ for unique pair $(m,b)$ with  $m \in \Z_+ $ and $0\le b < \rho$. 
We set the first (possibly empty set of) differences $x_{j+1}-x_j ,~j<m$ each to $\rho$. If $y-x$ is even we let $n=m$ or $n=m+1$ according to whether $b=0$ or not. In the latter case we let $x_{m+1}-x_m = b$. When $y-x$ is odd, we do as follows. If $b=1$ then set $n=m+1$ and let $x_n-x_{n-1} = 1$. Otherwise, let $n=m+2$ and set $x_{m+1}-x_m =b-1$ and $x_{m+2}-x_m=1$. 

If $y-x$ is even, then we obviously have

\be \label{eq:even_dist} d_t (x,y) \le \lfloor 1+ N/\rho  \rfloor \tilde d_t. \ee 

and when $y-x$ is odd, we have 

\be
\label{eq:new_triangle} 
 d_t (x,y)  \le \lfloor 1+N/\rho \rfloor \tilde d_t + d_t(x_n-1,x_n). 
 \ee
 
It remains to find an upper bound for $d_t(x_n-1,x_n)$. Choose $A$ such that 

\be d_t (x_n-1,x_n) = P_{x_n-1} (X_t \in A) - P_{x_n} (X_t \in A).\ee

We construct a coupling $(X,Y)$ starting from $(x_n-1,x_n)$ as follows.  Let $L$ and $R$ be $\nu_0$ and $\nu_N$ distributed random variables, respectively. We will assume that $L$ and $R$ are independent. For the first step we toss two independent fair coins, independent of $L$ and $R$.  If the first lands $H$, then $X$ moves and $Y$ stays put. Otherwise, $Y$ moves and $X$ stays put. If the second lands $H$ then we move the copy we chose one step to the right, meaning redistribution to $R$ if it's $Y$ and $Y$ is at $N$. If it lands $T$ then we move one step to the left, meaning redistribution to $L$  if it is $X$ and $X$ is at $0$. After this first step, both copies continue to evolve independently. Note that in any case, exactly one copy moves. If the copy moved is not redistributed then $Y_1-X_1 \in \{0,2\}$. If $X$ is redistributed from $0$ to $L$,  then $x_n=1$ and so  $X_1 -Y_1 = L-1$, so that the distance after one step is even. Similarly, if $Y$ is redistributed form $N$, then $x_n = N$, so that $X_1-Y_1 = N-1- R$, which is again even. \\

By the Markov property, 
\be d_t (x_n-1,x_n) = E_{(x_n-1,x_n)} \left (  \ch_A (X_t) - \ch_A (Y_t) \right) = E_{(x_n-1,x_n)} E_{(X_1,Y_1)}( \ch_A (X_{t-1})-\ch_A (Y_{t-1}) ).\ee

However, by the argument above and the triangle inequality, it follows from \eqref{eq:even_dist} that 

\be E_{(X_1,Y_1)}( \ch_A (X_{t-1})-\ch_A (Y_{t-1}) )\le \lfloor 1 + N/\rho \rfloor \tilde d_{t-1}.\ee

So that $d_t(x_n-1,x_n) \le \lfloor 1+ N /\rho  \rfloor \tilde d_{t-1}$. Plugging this into \eqref{eq:new_triangle} completes the proof. 

\end{proof} 

Next we recall a well-known classical result that will serve for estimating the coupling time. Suppose that $Z=(Z_n:n\in \Z_+)$ is the  lazy random walk on $\Z$, that is, $Z$ jumps to a neighboring site with probability $\frac 14$, and stays put with probability $\frac 12$. Write $Q_z$ for the distribution of $Z$ starting from $Z_0=z$. Let $T(L)$ denote the exit time of $Z$ from the set $\{1,\dots,L\}$, and let 

\be \lambda (L) = \frac 12 ( \cos ( \frac{\pi}{L+1}) +1). \ee 

 Then we have the following standard  and well-know lemma, whose proof is given in the Appendix. 
 
\begin{lemma}~
\label{lem:couple} 
\begin{enumerate} 
\item There exists a coupling  $(Z,Z')$ such that $Z'$ is lazy random walk starting from $Z_0=\lfloor (L+1)/2\rfloor$ and $T(L)\le T'(L)$, and where $T'$ is the exit time of $Z'$ from $\{1,\dots,L\}$. 
\item For any $z\in\{1,\dots,L\}$, 
\be Q_z (T(L) > t) =\frac{2}{L+1}\cot (\frac{\pi}{2(L+1)}) \sin ( \frac{\pi}{L+1} z) \lambda (L) ^t + \lambda_2(L)^t O(1),\ee
where $|\lambda_2(L)|<\lambda(L)$. 
\end{enumerate} 

\end{lemma} 
In light of the  Lemma \ref{lem:couple}-(1), we abuse notation and write $T(L)$ for the distribution of the exit time of $Z$ from $\{1,\dots,L\}$ starting from $\lfloor (L+1)/2\rfloor$.

\subsection{Deterministic Redistribution} 
\label{sec:deterministic}
In this section we will assume  in addition to \eqref{eq:parity_check} that $\nu_0$ and $\nu_N$ are deterministic. Specificially  

\be \nu_0=\delta_{J_0} \mbox{ and }\nu_N=\delta_{J_N},\mbox{ where }J_0,J_N \in \{3,5,\dots,N-3\}.\ee

Let 

\be 
\label{eq:L0} 
L_0=\frac 12 \max \{J_0-1,N-1-J_N, N+J_N-J_0\}.
\ee

Observe that $L$ is a positive integer, and will serve as the ``effective length" that will determine the exponential tail of the coupling time. Roughly speaking, $L_0$ is the longest interval one copy of the process needs to exit before the two copies meet, a sort of ``bottleneck" for the coupling. Since the coupling is efficient, this actually describes worst case scenario for convergence to stationarity. The geometric meaning of $L_0$ is as follows. Thinking of our state space as consisting of two ``loops" (which is incorrect), one from $0$ to $J_0$ (where we identify $J_0$ with $-1$), and the other from $J_N$ to $N$ (where we identify $J_N$ with $N+1$), then the first two listed elements of the set on the righthand side represent the lengths of the respective loops. The third, divided by $2$,  can be viewed as  distance between the centers of the loops.\\

Observe that the largest distance between the centers increases as $J_0$ decreases and $J_N$ increases, and attains a maximum of $N-2$ for  $J_0=3$, $J_N=N-3$. The minimal distance is $0$,  attained when   $J_0=N-3$ and $J_N=3$. We also observe the following additional bounds for $L_0$. 

\begin{enumerate} 
\item $L_0= \frac{N}{2}$ whenever  $J_0=J_N$  
\item $L_0\le N-3$ (attained when  $J_0=3, J_N=N-3$)
\item $L_0 \ge \frac{2 (N-1)}{3}$ attained when $J_N = \frac{N-1}{3}$. 
\end{enumerate} 

The main result of this section is the following: 

\begin{theorem}
\label{th:upper}
Suppose that $0\le x <y\le N$ and $y-x \in \{2,4,\dots,\rho\}$. Then there exists a  Markovian coupling with $(X_0,Y_0)=(x,y)$ such that the coupling time $\tau$ 
is  dominated by $\lfloor 6+ N /(\rho+1)\rfloor$ independent copies of $T(L_0)$. 
\end{theorem} 

In fact, the bound in the statement is weaker than the actual result proved, as the coupling time is dominated by a sum of independent random variables all dominated by $T(L_0)$. The number of the random variables as well as their distributions depend on $x,y$. In particular, with the exception of $x,y$ as in Stage 2c in our proof of the coupling, the coupling time is dominated by the independent sum of $T(\max(\frac {J_0-1}{2},\frac{N+1-J_N}{2}))+T(\frac{N+J_N-J_0}{2})+T(\frac{J_0-3}{2})$. We will not pursue this further because our main goal is obtaining the exponential rate, the statement will become messy, and all possibilities are obtained easily from the proof and   Figure 1. We comment however, that for $x,y$ an Stage 2c, and when $L_0 = \frac{N+J_N-J_0}{2}$, the bound obtained in our construction does contain a sum of at least two independent copies of $T(L_0)$, which implies that the coupling time decays exponentially, but with a polynomial tail. We do not know whether this is an artifact of our construction or  a limitation on Markovian couplings.\\

Since by Lemma \ref{lem:couple},  $T(L_0)$ has an exponential tail, and a finite sum of IID random variables with an exponential tail also has an exponential tail with the same exponent, it follows from the theorem, Proposition \ref{pr:triangle} and Lemma \ref{lem:couple}-(2) that 

\begin{corollary}
\be \limsup_{t\to\infty}  d_t \le \ln \lambda (L_0).\ee   
\end{corollary} 

To show that the coupling constructed in Theorem \ref{th:upper} is efficient, we need a matching lower bound. Here it is: 

\begin{prop}
\label{pr:lower}
\be d_t \ge \frac 12 (1- \frac{2\pi}{N} )\lambda(L_0)^t.\ee 
\end{prop} 

\subsection{Random Redistribution} 
Here we will consider more general redistribution measures, but under the additional assumption that $\nu_0=\nu_N$.\\

As seen in the last section, if $\nu_0=\nu_N=\delta_{J_0}$,  then $\frac 1t \ln d_t \sim \ln \lambda (\frac{N}{2})$ as $t\to\infty$, independently of the choice of $J_0$. In this section we show that this remains  the same under the present, more general, assumptions. The analogous results for Brownian motion instead of lazy random walk  were first obtained by  Li and his coauthors  \cite{LLR} through Fourier analysis, and,  a probabilistic proof using coupling was given  \cite{KW}. The coupling we present here is an adaptation of the coupling idea from the latter work, and we present it here because we want to distinguish it from the case of the previous section. 
\begin{theorem} 
\label{th:sym_upper}
Suppose that $0\le x <y\le N$ and $d= y-x \in \{2,\dots,\rho \}$. Then there exists a coupling with $(X_0,Y_0)=(x,y)$ such that the coupling time $\tau$ is  dominated by at most $5$  independent copies of $T(N/2)$. 
\end{theorem}

The matching lower bound is given by 
\begin{prop}
\label{pr:sym_lower}
$d_t \ge P(T(N/2)>t)$.  
\end{prop}

We highlight the following with regard to Theorem \ref{th:sym_upper}: 

\begin{enumerate} 
\item 

As shown in our proof of Theorem \ref{th:sym_upper},  the coupling is not Markovian, unless $\nu_0$ is a point mass distribution,  and this raises the question whether there does exist a Markovian coupling at all, and what is the best bound a Markovian coupling can give. The same questions are even more interesting for the case $\nu_0\ne \nu_N$ with none being a point-mass distribution. We leave these for future research. 
\item Unlike the coupling of Theorem \ref{th:upper}, the coupling in Theorem \ref{th:sym_upper} ends after at most 5 stages, independently of the parameters.
\end{enumerate}

\section{Proofs} 
\subsection{Coupling Regimes} 
\label{sec:coupling}
To prove the theorem, we begin by introducing the couplings we will apply. The main idea is to switch between two coupling regimes, according to the state of the system. It is convenient and simpler to describe the coupling using simple symmetric lazy random walk on $\{-1,0,\dots,N,N+1\}$, with transition of the random walk from $0$ to $-1$ identified with redistribution to $J_0$, and transition from $N$ to $N+1$ identified with  redistribution to $J_N$. Here are the two coupling regimes to be employed:
\begin{enumerate}
\item {\it Rigid} coupling.  The increments of the two copies are identical. 
\item {\it Reflection (Ref)} coupling. The increments of the two random walks are opposite. 
\end{enumerate} 


Switching between the two regimes occurs at hitting times of the joint process. 

\subsection{Deterministic Redistribution}
\begin{proof}[Proof of Theorem \ref{th:upper} ]
Without loss of generality we assume  that $J_0\le N-J_N$. 

In order to simplify the description of the coupling, we let $\ol{X}_t= \min (X_t,Y_t)$ and $\ol{Y}_t = \max (X_t,Y_t)$. \\ 

Suppose that $\ol{X}_0 =x$ and $\ol{Y}_t=y$, and let $D=y-x$. Then $D$ is even and $D< J_0$.  Define the ``symmetric" points: 
  
\begin{equation} 
\label{eq:symmetric} 
 \ell_0(D) =  \frac{ J_0 -1 - D}{2}\mbox{ and } \ell_N(D) =  \frac{ N+1+J_N-D}{2}.
 \end{equation} 

The coupling is done in four stages. We begin from the stage that corresponds to the initial condition of the system. In the description below, the first item describes the initial configuration. \\

\noindent {\bf Stage 1.} 
\begin{enumerate} 
\item $\ol{X}_0 \in \{ \ell_0(D),\ell_N(D)\}$. Apply Ref coupling. 
\item Stop when coupling occurs. 
\item Time to complete: $T(\frac{J_0 -1}{2})$ if $\ol{X}_t = \ell_0(D)$ and $T(\frac{N - J_N-1}{2})$ if $\ol{X}_t = \ell_N(D)$. 
\end{enumerate} 

When $\ol{X}$ does not begin from either of the symmetric points, we will drive it to one of them. This will be done through rigid coupling. If $\ol{X}$ is between the symmetric points we apply rigid coupling (Stage 2a). When $\ol{X}_t < \ell_0(D)$ or $\ol{X}_t > \ell_N(D)$, there  may be redistribution before $\ol{X}$  reaches  either point (Stages 2b and 2c).  

\noindent{\bf Stage 2a.} 
\begin{enumerate} 
\item $\ol{X}_0 \in \{\ell_0(D)+1,\dots, \ell_N (D)-1\}$ and $D < J_0$. Apply rigid coupling. 
\item Stop when $\ol{X}$ hits $\{\ell_0(D),\ell_N(D)\}$.  
\item Time to complete: $T(\ell_N(D) - \ell_0(D)-1)=T( \frac{N+ J_N -J_0}{2})$. 
\end{enumerate} 

When stage ends we continue to stage 1. 

\noindent {\bf Stage 2b.} 
\begin{enumerate} 
\item $\ol{X}_0 < \ell_0(D)$ and $D < J_0$. Apply rigid coupling. 
\item Stop when either
\begin{enumerate} 
\item $\ol{X}$ hits $\ell_0(D)$; or
\item $\ol{X}$ is redistributed from $0$.
\end{enumerate} 
\item Time to complete: $T(\ell_0(D) )=T( \frac{J_0-1-D}{2})$. 
\end{enumerate} 
If the first alternative holds, we continue to Stage 1. Otherwise, at the end of the stage, $\ol{Y}$ is at $J_0$ while $\ol{X}$ is at $D-1$. Thus, the new distance is $D' = J_0 +1 - D$, which is even because $J_0+1$ and  $D$ are  even, and $D'<J_0$ because $D$ is an even positive integer. Observe then that 
$$\ell_0(D') =  \frac{ J_0-1-D'}{2} = \frac{ J_0-1 -(J_0+1-D)}{2}=\frac{D}{2}-1<D-1.$$
 Therefore the symmetric point is below the position of $\ol{X}$, and we continue to stage 2a, with the new distance $D'$. 

\noindent{\bf Stage 2c.} 
\begin{enumerate} 
\item $\ol{X}_0 > \ell_N (D)$ and $D < J_0$. Apply rigid coupling. 
\item Stop when either
\begin{enumerate} 
\item $\ol{X}$ hits $\ell_N(D)$; or 
\item $\ol{Y}$ is redistributed from $N$. 
\end{enumerate} 
\item Time to complete: $T(N - \ell_N(D)-D)=T( \frac{N-J_N-1-D}{2})$. 
\end{enumerate} 
If the first alternative holds, we continue to Stage 1. Otherwise, at the end of the stage, $\ol{X}$ is at $J_N$ while $\ol{Y}$ is at $N+1 - D$. Thus, the new distance is $D' =N+1-D -J_N$. If $D'\ge J_0$, we continue to Stage 3. Otherwise, it is clear that $\ol{X} \le \ell_N(D')$, so we  continue to Stage 2a, with the new distance $D'$. 

\noindent{\bf Stage 3.} 
\begin{enumerate} 
\item $\ol{X}_0 \in \{J_0,J_N\}$, $\ol{X}_0 \leq{\ell_N(D)}$, and $\ol{Y}_0 = \ol{X}_t+D$ with $D\ge 2$ even and $D \ge J_0$. Apply rigid coupling.
\item Stop when either 
\begin{enumerate} 
\item $\ol{X}$ hits $\ell_N (D)$; or 
\item  $\ol{X}$ is redistributed from $0$. 
\end{enumerate} 
\item Time to complete: $T(\ell_N(D) )$.  Since $D\ge J_0$, $J_0$ is odd and $D$ is even, it follows from \eqref{eq:symmetric} that the time is dominated by $T( \frac{N+J_N - J_0}{2})$. 
\end{enumerate} 
If the first alternative holds, we continue to Stage 1. In this case, we adapt Stage 1 slightly, as we now have $D' > J_0$. However, since $\ol{X}_t = \ell_N(D')$ and $\ol{X}_t < \ol{Y}_t$ by definition, $D' < N+1 - J_N$, and so Stage 1 works as before.

Otherwise, by assumption, after redistribution we have that $\ol{X} = J_0$ and $\ol{Y} = D-1\ge J_0$, so the new distance $D' = D- (1+J_0)$. Let us consider three  alternatives: 
\begin{enumerate} 
\item $\ol{X} = J_0> \ell_N(D')$. This can only occur if 2c started from $J_N$ (and only if $J_0>J_N$). In this case, we continue to Stage 4.
\item $\ol{X} = J_0 = \ell_N (D')$. Then we continue to Stage 1, as from alternative 2a. 
\item $\ol{X}=J_0 < \ell_N (D')$. If $D'<J_0$, we continue to stage 2a. Otherwise, we iterate stage 3. Since in each iteration the distance decreases by $1+J_0$, the number of iterations does not exceed  $\lfloor D/ (1+J_0)\rfloor$, and we continue Stage 1 or to Stage 2a. Note that since the distance decreases after each iteration, $\ell_N$ increases after each iteration. 
\end{enumerate}
 
\noindent{\bf Stage 4}
\begin{enumerate}
\item $\ol{X}_0= J_0 > \ell_N(D)$ and $\ol{Y}_0 = J_0 + D$. Apply Ref coupling. 
\item Stop when either
\begin{enumerate}
\item Coupling occurs, or
\item $\ol{Y}$ is redistributed from $N$. 
\end{enumerate}
\item Time to complete: $ T( N- J_0 - D/2) $ but since $J_0 > \ell_N (D)$, it follows from \eqref{eq:symmetric}  that $N-J_0-D/2 < \frac{N-1-J_N}{2}$, so time dominated by $T(\frac{N-1-J_N}{2})$.  
\end{enumerate}
If the second option occurs, $\ol{X} = J_N$ and $\ol{Y} = J_0 - (N+1-(J_0 + D))$. Since $ J_N < \ol{Y} < J_0$, $D' = \ol{Y} - \ol{X} < J_0$. Thus, depending on the relation of $\ol{X} = J_N$ to $\ell_0(D')$, we continue to either Stage 1, 2a, or 2b. Furthermore, since, initially, $\ol{X} > \ell_N(D)$, we have that $\ol{X} - J_N > N+1 - \ol{Y}$ and in particular, $\ol{X} > N+1-\ol{Y}$. Thus, we disregard the possibility of $\ol{X}$ redistributing from 0 as $\ol{Y}$ would always be redistributed first. \\

\setcounter{figure}{0}
\begin{figure}[h!]
\includegraphics[scale=0.8]{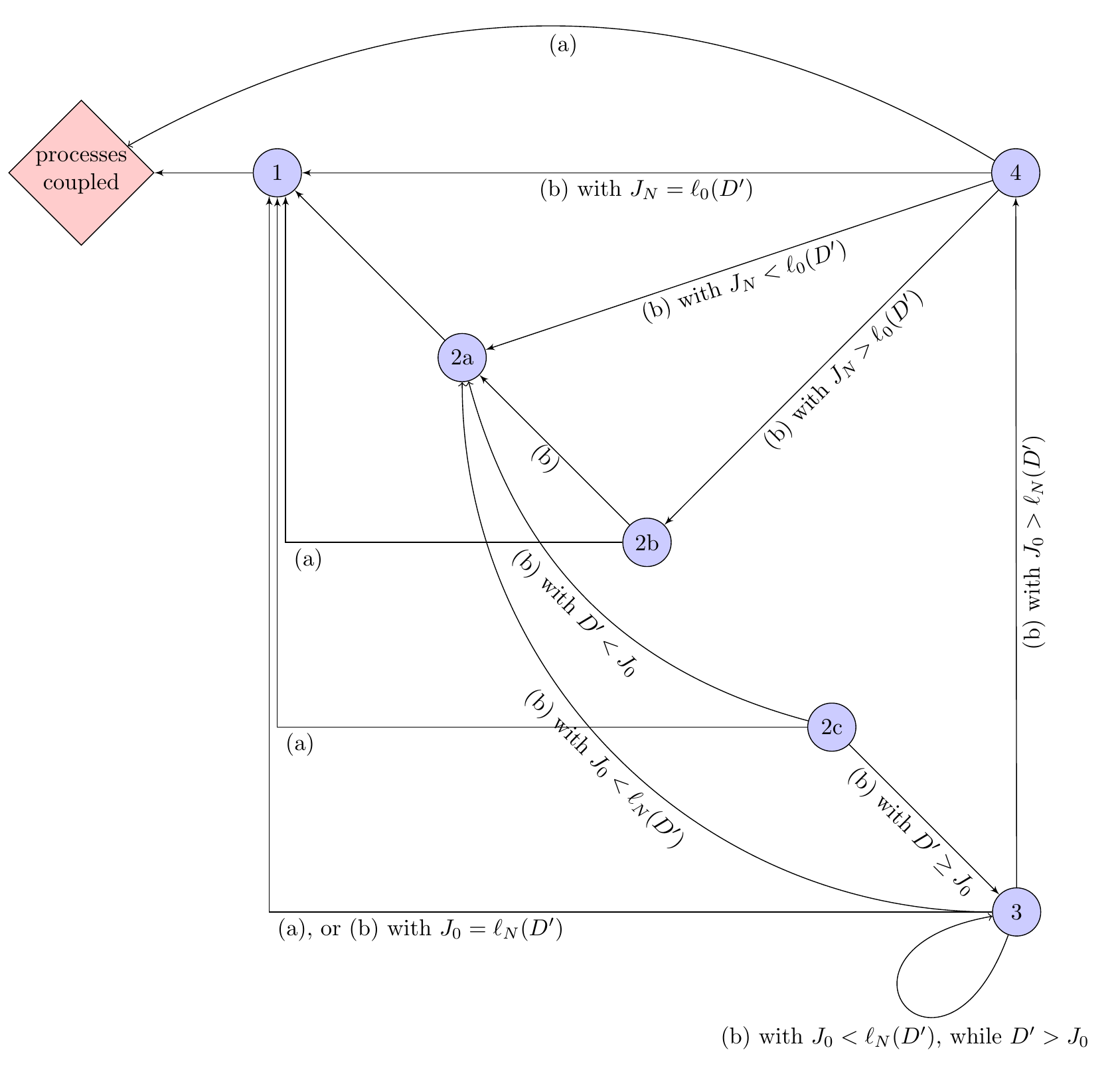}
\caption{Summary of the coupling from Section \ref{sec:coupling}} 
\label{fig:stages}
\end{figure} 

Let us review the coupling. Figure \ref{fig:stages} displays all possible implementations of the coupling.  Stages 1 and 2a,2b,2c are the initial steps in the sense that the coupling must begin from one of them.  In the coupling, each of these steps is repeated at most once. Stage 2c is special in the sense that it takes care of the case that the redistribution may lead to a distance bigger or equal to $J_0$. Stage 3 is invoked when this happens. \\

Stage 3 may be repeated a number of times, bounded above by $\lfloor 1+N/(J_0+1)\rfloor$. From Stage 3, the coupling proceeds to either stage 1,2a or 4. From Stage 4, the coupling can either end (meeting) or continue to one of the Stages 1,2a,2b.\\

From the point of view of duration of the coupling, the meeting time is bounded above by the independent sum of times for all Stages, with possible repetitions for Stage 3. Listing the times for each of the stages in order of appearance, we have $T((J_0-1)/2)$ or $T( (N-J_N-1)/2)$ (Stage 1), and 
$T(\frac{N+J_N-J_0}{2})$, $T(\frac{J_0-1-D}{2})$ and $T(\frac{N-J_N-1-D}{2})$ (Stage 2), $T(\frac{N+J_N-J_0}{2})$ (Stage 3) and $T(\frac{ N-J_N-1}{2})$ (Stage 4). The maximal length among all intervals mentioned above is therefore $\frac 12 \max \{ J_0-1,N-J_N-1, N+J_N-J_0\}$, which is $L_0$.   Finally, Figure \ref{fig:stages} shows that the maximal number of steps (omitting Step 3) is $5$, and the result now follows. 
\end{proof} 

\begin{proof} [Proof of Proposition \ref{pr:lower}]
We will find an eigenfunction $f$  for $p$ which is not constant, and has a  real eigenvalue $\lambda$. Without loss of generality, we may assume that $\sup |f| \le 1$. For any $x,y$,   $d_t (x,y) \ge \frac 12 \left ( E_x f (X_t) - E_y f (X_t) \right)$.  Since we must have $\sum_x f(x) \pi(x) =0$, where $\pi$ is the stationary distribution of $X$,  it follows that $f$ attains both strictly positive and strictly negative values. In particular, we can choose $x$ and $y$ such that $f(x) > 0 > f(y)$, and it immediately follows that 
\begin{equation} 
\label{eq:lower} 
d_t (x,y) \ge \frac 12 ( f(x) - f(y)) \lambda^t.
\end{equation}
 In order to find $f$, we will choose $f(x) =\sin (\rho x + \omega)$, and will find choices for the parameters $\rho$ and $\omega$ that match the upper bounds from Theorem \ref{th:upper}.  Suppose now that $x \in \{1,\dots,N-1\}$. Then 
$$ p f(x) =  \frac 14 \left  ( \sin (\rho (x+1) + \omega) + \sin (\rho (x-1) + \omega) \right) + \frac 12 \sin (\rho x + \omega).$$ 
Thus, $p f (x) = \frac 12 (\cos \rho +1) f(x)$. Now if we extend $f$ to $\{-1,\dots, N+1\}$ and additionally impose the constraints 
 \be 
 \label{eq:ev_eqn}  \begin{cases} \sin ( -\rho +\omega) = \sin (\rho J_0 + \omega) \mbox{ and }\\ \sin (\rho (N+1) + \omega) = \sin ( \rho J_N + \omega),\end{cases}
 \ee
then it immediately follows that $f$ (restricted to the state space) is indeed an eigenfunction for $p$ with eigenvalue $\lambda =\frac 12 (\cos \rho +1)$.  In order to proceed, we need to find choices for $\rho$ and $\omega$ that will satisfy the constraints. \\

The first constraint is met if $-\rho + \omega +2\pi = \rho J_0  + \omega$. That is $\rho (J_0+1) = 2\pi$, or $\rho = \frac{2\pi }{J_0+1}$. With this choice of $\rho$, the first constraint is met for all choices of $\omega$, which allows to freely chose $\omega$ to meet the second constraint. The actual value of $\omega$ is irrelevant for the eigenvalue calculation. \\

As is easy to see, we can repeat the argument by considering the second constraint first. This will give us $\rho  =\frac{2\pi}{N+1 -J_N}$. \\

Another way to satisfy the  first constraint  is to have the arguments of the $\sin$ function in the  equation symmetric with respect to a maximum or a minimum of the $\sin$ function, that is $\pi/2 + \pi k$ for some integer $k$. In taking $k=0$, the first constraint is satisfied when 
$$ \frac \pi2 -( -\rho +\omega)  = \rho J_0 + \omega - \frac \pi 2,$$
and the second constraint will be met if 
$$ \frac {3\pi}{2} - (\rho J_N + \omega) = \rho (N+1) +\omega  - \frac {3\pi}{ 2}.$$ 
Subtract the first equation from the second to obtain 
$\pi - \rho (J_N +1) = -\pi + \rho (N+1-J_0)$, that is 
$$ \rho = \frac{2\pi}{N+1 +J_N - (J_0 -  1)}.$$
Using the first equation, $\omega = \frac {\pi}{2} -\frac { \rho}{2}  (J_0-1)$, and the second equation is  satisfied too. \\

In light of the above, we see that  \eqref{eq:lower} is satisfied when $\lambda$ is chosen to be $\lambda_1$ in the statement of the theorem, and $f$ is the corresponding eigenfunction of the form $\sin (\frac{\pi}{L_0+1} x + \omega_1)$. Observe  that $\rho_1 \ge \frac{\pi}{N}$. In particular the set $\{\omega_1,\rho_1 +\omega, 2\rho_1+ \omega, \dots , N \rho_1 + \omega\}$ contains at least one element whose distance from  $\pi/2 + \pi k_1$ for some $k_1 \in \Z$,  is at most $\frac{\pi}{N}$, as well as an element whose distance  from  $\pi k_2$ for some $k_2 \in \Z$ is at most $\frac {\pi}{2}$. Call the first $x$ and the second $y$. Without loss of generality, we may assume $f(\pi/2 + k \pi)=1$ (otherwise change to $-f$). It follows that 
$f(x) \ge 1- \frac{\pi}{N}$. Similarly, $f(y)\le \frac{\pi}{N}$. Therefore, $f(x) - f(y) \ge 1- \frac{2 \pi}{N}$, and the result follows.
\end{proof}

\subsection{Random Redistribution} 
\begin{proof}[Proof of Theorem \ref{th:sym_upper}]
We first prove the theorem for the case that $x,y\le N/2$. By symmetry, this also cover the case where $x,y \ge N/2$. After we construct the coupling for this stage, we extend it to the remaining case $x\le N/2$ and $y>N/2$.\\

Assume then that $0 \le x < y \le N/2$ and that $d= y-x \le \rho$.  Let $K,K'$ be two independent random variables  distributed according to $\nu_0$.\\

\noindent{\bf Stage 1a.} 
\begin{enumerate} 
\item $X_0 = x, Y_0= y,~0\le x<y \le N/2$, $d\le \rho$. Apply Rigid coupling. 
\item Stop when either 
\begin{enumerate} 
\item $Y$ hits $N/2$, then continue to 1b; or 
\item $X$ is redistributed from $0$ to $K$. Continue to Stage 2a. 
\end{enumerate} 
\item Time is bounded above by $T (N/2)$. 
\end{enumerate} 

\noindent 
{\bf Stage 1b.}
\begin{enumerate} 
\item $Y_0= N/2$ and $X_0 =N/2-d$. Apply Ref coupling. 
\item Stop when either 
\begin{enumerate} 
\item Copies meet; or 
\item $Y_t- X_t =\rho $. Continue to 1c. 
\end{enumerate} 
\item Time is bounded above by $T(\rho/2-1)$. 
\end{enumerate} 
Note that the second alternative will occur before a redistribution, because when $X$ hits $0$, the distance will be $N/2-d + d + N/2-d = N-d$, and since $d \le \rho$ this quantity is greater or equal to $\rho$. \\

\noindent 
{\bf Stage 1c.}
 \begin{enumerate} 
\item $X_0 = N/2-d - (\rho-d)/2,~Y_0 = N/2+(\rho-d)/2$.  Apply Rigid coupling. 
\item Stop when either 
\begin{enumerate}
\item $Y_t = (N+\rho)/2$, then continue to Stage 3; or  
\item $X$ is redistributed from the origin to $K$, then continue to Stage 3. 
\end{enumerate}
\item Time is bounded above by $T (\frac{N-\rho}{2}+1)$. 
\end{enumerate} 
If the second alternative holds and $K\le N/2$, the processes meet. Otherwise after this stage ends the copies are at $a$ and $N-a$ for some $a$. \\

\noindent 
{\bf Stage 2a.} $X_0=K$ and $Y_0=d-1$ ($Y$ never jumped). Apply Ref coupling. 
\begin{enumerate} 
\item Stop when either 
\begin{enumerate} 
\item Copies meet; or 
\item When distance is $K+1$ and continue to 2b. 
\end{enumerate} 
\item Time bounded above by  $T(d/2)$. 
\end{enumerate} 

\noindent 
{\bf Stage 2b.} 
\begin{enumerate} 
\item $X_0=K+d/2,Y_0=d/2-1$ ($Y$ never jumped). Apply Rigid coupling. 
\item Stop when either:
\begin{enumerate} 
\item $Y$ is redistributed from $0$ to $K$, and copies meet; or 
\item $Y_t= N/2 - (K+1)/2$, then continue to stage 3. 
\end{enumerate} 
\item Time is bounded above by $T(N/2- (K+1)/2)$. 
\end{enumerate} 

\noindent {\bf Stage 3.} 
\begin{enumerate} 
\item $Y_0=N-X_0$. Apply Ref coupling.  
\item Stop when either 
\begin{enumerate} 
\item The copies meet; or 
\item The copies are redistributed from $0$ and $N$ simultaneously to $K'$. 
\end{enumerate} 
\item Time is dominated by $T(N/2)$. 
\end{enumerate} 

As is easy to see, the coupling ends after no more than 4 Stages, the longest chain being 1a$\to$ 1b$\to$1c $\to$  3, and the times for the stages are all dominated by IID copies of $T(N/2)$. Furthermore, this coupling is not Markovian, because the in Stage 2b, if $Y$ is redistributed from $0$, the transition is from $(0,K+1)\to (K,K)$, whereas, if the coupling were Markovian and $\nu_0(y) < 1$, then for any $y<N$, there  exists $x\ne y$ and $y'$ such that  the transition  $(0,y+1)\to (x,y')$ occurs with positive probability. \\

We have therefore completed the proof for the case where $x,y$ are both $\le N/2$. Suppose now that $x\le N/2$ and $y > N/2$. Let $\tilde y=N-2$. Then $x,\tilde y\le N/2$, and furthermore, since $0<y-x \le \rho$, and $y-x = |y-N/2| + |N/2-x|= |\tilde y - N/2| + |N/2-x|$. But $|\tilde y- x| \le \max \{N/2-\tilde y,N/2-x\}\le y-x\le \rho$. Furthermore, since $\tilde y -y$ is even, $\tilde y -x$ is even too. Thus, we can construct a coupling $(X,\tilde Y)$ starting from $(x,\tilde y)$ whose coupling time will be dominated by at most 4 independent copies of $T(N/2)$. However, letting $Y = N - \tilde Y$, at the coupling time $\tau$ for $(X,\tilde Y)$, we have that $Y_{\tau} =N -X_{\tau}$. Therefore, applying stage 3 again (with an independent copy of $K'$) guarantees that $X$ and $Y$ will be coupled after no more than $5$ independent copies of $T(N/2)$. 

\end{proof} 

\begin{proof}[Proof of Proposition \ref{pr:sym_lower}]
  Let $f(x) = 1$ if $x\le N/2$ and $f(x)=0$ otherwise. Now 
 $$ d_t \ge E_{N/4} f (X_t) - E_{3N/4} f (Y_t) =E_{(N/4,3N/4)} (f (X_t) - f (Y_t)),$$
 for every coupling $(X,Y)$. If we choose Ref coupling, until the copies meet and then the move together, and since $3N/4-N/4$ is even, it follows that the meeting time $\tau$ occurs exactly when $X$ exists the set $\{0,\dots,N/2-1\}$. In particular,  
  \be d_t \ge E_{(N/4,3N/4)}\left [ \left (  f (X_t) - f (Y_t)\right) ;\tau>t \right],\ee 
  and the righthand side is equal to  $P( T(N/2)>t)$. 
  \end{proof} 
  
\section{Appendix} 
\begin{proof}[Proof of Lemma \ref{lem:couple}] ~\\
\noindent 1. Suppose $L$ is odd so that $Z'$ starts from the unique center $\frac{L+1}{2}$ and $Z$ starts from $z\in\{1,...,L\}$. Assume WLOG that $z<\frac{L+1}{2}$. If $\frac{L+1}{2}-z$ is even, we run an Ref coupling until $Z$ exits or $Z$ and $Z'$ meet at $z+\frac{1}{2}(\frac{L+1}{2}-z)$ and continue together until exiting. In either case $T(L)\leq  T'(L)$. If $\frac{L+1}{2}-z$ is odd, the method outlined in the description of rigid coupling can be employed. Namely, we toss two fair coins, $A$ and $B$, and have $A$ determine whether $Z$ or $Z'$ moves while $B$ determines which way it moves. After this procedure, either $Z$ will have exited or $Z'-Z\geq 0$ is even in which case we can continue with the Ref coupling as before. Now suppose that $L$ is even. In this case there are two centers, $\frac{L+1}{2}\pm\frac{1}{2}$. Assume WLOG that $z<\frac{L+1}{2}-\frac{1}{2}$. If $\frac{L+1}{2}-\frac{1}{2}-z$ is odd then have $Z'$ start at $\frac{L+1}{2}+\frac{1}{2}$, otherwise have $Z'$ start at $\frac{L+1}{2}-\frac{1}{2}$. This ensures that $Z'-Z$ is even so we can run the Ref coupling as before. For a given $z$, this only establishes that $T(L)\leq T'(L_0)$ for $Z'$ starting from a particular center. Hence the exit time starting from $z$ is stochastically dominated by the exit time starting from that particular center. However, since the exit times starting from either center have the same distribution, this is of no consequence and the stochastic dominance holds for both centers.\\

\noindent 2. Observe that the sub-Markovian transition function for $Z$ killed outside $\{1,\dots,L\}$ is symmetric and irreducible on $\{1,\dots,L\}$. From Perron-Frobenius, the largest eigenvalue is simple. The corresponding eigenvector, the Perron eigenvalue  is (WLOG) strictly positive on $\{1,\dots,L\}$.  By symmetry there exists an orthonormal basis with respect  to the counting measure consisting of eigenvectors, including the normalized Perron eigenvector as an element. It follows that the Perron root is the unique element in the basis which does not change signs. Finally, note that the function $\sin ( \frac{\pi}{L+1}x)$ on $\{1,\dots,L\}$ and zero elsewhere is an eigenfunction for this transition function, strictly positive on $\{1,\dots,L\}$. Thus this a Perron eigenvector. In addition, as is easy to see, the eigenvalue is $\lambda(L)=\frac 12 (\cos  \frac{\pi}{L+1} +1)$. A straightforward computation shows that the $\ell^2$ normalized Perron eigenfunction is $ \sqrt{ \frac 2 {L+1}}\sin ( \frac{\pi}{L+1}x)$, and that $\sum_{x=1}^{x=L}\sin ( \frac{\pi}{L+1} x) = \cot (\frac{\pi}{2(L+1)})$. The result now follows from the spectral theorem. 
\end{proof}


\bibliographystyle{amsalpha}
\bibliography{rwcoupling}

\end{document}